\documentclass[noamsfonts]{amsart}
\usepackage[charter,expert]{mathdesign}
\usepackage[osf]{XCharter}
\selectfont
\usepackage[activate={true,nocompatibility},final]{microtype}

\usepackage[all, arc]{xy}
\SelectTips{eu}{}
\entrymodifiers={+!!<0pt,\fontdimen22\textfont2>}

\usepackage[pdftex, colorlinks=true, linkcolor=blue, citecolor=magenta, linktocpage]{hyperref}
\theoremstyle{plain}
\newtheorem{thm}[subsubsection]{Theorem}
\newtheorem{lemma}[subsubsection]{Lemma}
\newtheorem{prop}[subsubsection]{Proposition}
\newtheorem{cor}[subsubsection]{Corollary}

\theoremstyle{definition}
\newtheorem{example}[subsubsection]{Example}

\theoremstyle{definition}

\def\cC{\mathcal{C}}

\def\cX{\mathcal{X}}
\def\cY{\mathcal{Y}}
\def\cZ{\mathcal{Z}}

\def\11{\mathbf{1}}

\def\CC{\mathbf{C}}

\def\ZZ{\mathbf{Z}}

\def\Hom{\mathrm{Hom}}

\def\id{\mathrm{id}}

\def\Spec{\mathrm{Spec}}

\def\Map{\mathrm{Map}}
\def\PreGrpd{\mathrm{PreGrpd}}

\def\Grpd{\mathrm{Grpd}}

\def\L{\mathbf{L}}
\def\R{\mathbf{R}\!}

\def\E{\mathbb{E}}
\def\B{\mathbb{B}}

\def\Aut{\mathrm{Aut}}

\def\Ho{\mathrm{Ho}}

\newcommand{\mapright}[1]{\xrightarrow{#1}}

\newcommand{\htimes}{\times^h}
\makeatletter
\newcommand{\holim@}[2]{%
      \vtop{\m@th\ialign{##\cr
                  \hfil$#1\operator@font holim\,$\hfil\cr
                      \noalign{\nointerlineskip\kern1.5\ex@}#2\cr
                          \noalign{\nointerlineskip\kern-\ex@}\cr}}%
                      }
                      \newcommand{\holim}{%
                            \mathop{\mathpalette\holim@{\leftarrowfill@\textstyle}}\nmlimits@
                        }
                        \makeatother

                        \title{Homotopy limits and fixed point stacks}
\author{R. Virk}
\address{The Appalachians}
\begin{document}
\maketitle
%\setcounter{tocdepth}{1}
%\tableofcontents
%\renewcommand{\thesubsection}{\textbf{\arabic{section}.\arabic{subsection}}}
%\renewcommand{\thesubsection}{\textbf{\arabic{subsection}}}
\renewcommand{\thesubsection}{\arabic{subsection}}
\subsection{Introduction}
We discuss homotopy limits of stacks in groupoids. Inertia stacks, the usual stack fibre product and fixed point stacks are all incarnations of these.
We examine the compatibility of homotopy limits with each other and with stackification. 

The \emph{statements} of the results make no mention of homotopy theory. However, the central organizing principle of this note, used explicitly in some of the proofs, is that the category of simplicial (pre)sheaves on a Grothendieck site may be endowed with a simplicial model structure such that fibrant objects satisfy descent. We do not work with simplicial sheaves though, focusing on (pre)sheaves of groupoids instead. Of course, groupoids may be embedded into the category of simplical sets via the nerve construction. 
However, groupoids have special features (their nerves are 2-coskeletal Kan complexes) not enjoyed by arbitrary simplicial sets. On a technical level this is exemplified by Lemma \ref{lem:key} whose na\"ive analogue for arbitrary simplicial sets is false.\footnote{Lemma \ref{lem:key} is essentially a statement about compact objects. In the category of simplicial sets these are given by objects with finitely many non-degenerate simplices. The nerve of a finite category may not satisfy this - eg. the classifying space of any non-trivial finite group. This leads to the so-called canonical mistake: the \emph{false} assumption that finite homotopy limits of simplicial sets commute with filtered colimits. However, this \emph{does} hold for groupoids.}

\subsection{Notation and conventions}
\subsubsection{}`Map', `morphism' and `arrow' are used interchangeably. 
\subsubsection{} `Canonical' is written in lieu of `a natural transformation of functors'. 
\subsubsection{} A map of categories means a functor of categories.
\subsubsection{} Given small categories $K$ and $G$, set
\[ \Map(K,G) = \text{category of functors $K\to G$}. \]
Morphisms are natural transformations. Note that the obvious canonical map
\[ \Map(K\times H, G) \to \Map(K, \Map(H, G)) \]
is an isomorphism (the so-called exponential law).
\subsubsection{} The symbol $\ast$ is reserved for the final object in a category (if it exists). 
\subsubsection{} The symbol $\cong$ is only used for isomorphisms. Other notions of equivalence (equivalence of categories, weak equivalence, etc.) are denoted by $\simeq$ or $\mapright{\sim}$. 
\subsubsection{} A \emph{finite category} is a category with finite sets of objects and morphisms.
\subsubsection{} A groupoid is a small category in which all maps are invertible. We set
\[ \Grpd = \text{category of groupoids}.\]
To emphasize, $\Grpd$ is a strict (1-)category (as opposed to a 2-category).
\subsubsection{} A map of groupoids is an \emph{equivalence} if it is an equivalence of categories. 
\subsubsection{} A map of groupoids $f\colon X\to Y$ is a \emph{fibration} if, for each $x\in X$ and each isomorphism $\alpha\colon f(x)\to y$ in $Y$, there exists an isomorphism $x\to x_1$ in $X$, such that $f(\beta) = \alpha$. A fibration that is also an equivalence is called \emph{acyclic}
\subsubsection{} Given a group $G$ and a $G$-set $X$, we write $\E_G X$ for the groupoid with set of objects $X$ and an arrow $x\mapright{g} y$ for each $g\in G$ such that $g(x) = y$.
We set
    \[ \E G = \E_G G \qquad \text{and} \qquad \B G = \E_G\{\ast\},\]
where $G$ acts on itself via multiplication.

\subsection{Homotopy limits: groupoids}
Let $\Gamma$ be a small category. For $\alpha\in\Gamma$, the overcategory $(\Gamma \downarrow \alpha)$ is the category whose objects are morphisms $\beta\to \alpha$ in $\Gamma$. Maps are given by the evident commutative triangles.
Let $X$ be a $\Gamma$-diagram in $\Grpd$ (i.e., $X$ is a functor $\Gamma\to\Grpd$). Its \emph{homotopy limit}, denoted $\holim_{\Gamma}X$, is the equalizer of 
\[ 
\xymatrix{
    \underset{\alpha\in \Gamma}{\prod} \Map((\Gamma\downarrow \alpha), X(\alpha))\ar@<-1ex>[r]_-{\psi}\ar@<1ex>[r]^-{\phi} & \underset{\{\alpha \to \beta\}\in \Gamma}{\prod}\Map((\Gamma\downarrow\alpha), X(\beta)),
} \]
where $\phi$ is induced by the natural maps
\[ \Map((\Gamma\downarrow \alpha), X(\alpha)) \to \Map((\Gamma\downarrow \alpha), X(\beta)) \]
given by composition with $X_{\alpha}\to X_{\beta}$ (for each $\alpha\to\beta$);
and $\psi$ is induced by the maps
\[ \Map((\Gamma\downarrow \beta), X(\beta)) \to \Map((\Gamma\downarrow \alpha), X(\beta)), \]
given by pre-composition with the evident map $(\Gamma\downarrow\alpha)\to(\Gamma\downarrow\beta)$ (for each $\alpha\to\beta$).
As all small limits exist in the category of groupoids (they are given by limits of sets on objects and morphisms), this equalizer always exists.
If $f\colon X\to Y$ is a map of $\Gamma$-diagrams (i.e., a natural transformation of functors), we write
\[ f_*\colon \holim_{\Gamma}X \to \holim_{\Gamma} Y\]
for the obvious induced map.
The usual limit of the $\Gamma$-diagram $X$ is the equalizer of
\[ 
\xymatrix{
    \underset{\alpha\in \Gamma}{\prod} X(\alpha)\ar@<-1ex>[r]\ar@<1ex>[r] & \underset{\{\alpha \to \beta\}\in \Gamma}{\prod} X(\beta),
} \]
where the maps are the analogues of $\phi$ and $\psi$ above. This may be rewritten as
\[ 
\xymatrix{
    \underset{\alpha\in \Gamma}{\prod} \Map(\ast, X(\alpha))\ar@<-1ex>[r]\ar@<1ex>[r] & \underset{\{\alpha \to \beta\}\in \Gamma}{\prod} \Map(\ast, X(\beta)).
} \]
So there is a canonical map $\varprojlim_{\Gamma}X \to \holim_{\Gamma}X$. This map is usually not an equivalence.

\begin{example}\label{eg:homotopypullback}
    Let $\Gamma$ be the pullback category. I.e., $\Gamma$ has three objects and non-identity maps depicted by the diagram $\bullet\to\bullet\gets\bullet$. A $\Gamma$-diagram amounts to maps $f\colon X\to Z$ and $g\colon Y\to Z$ in $\Grpd$. The corresponding homotopy limit is called the \emph{homotopy pullback} and is denoted $X\htimes_Z Y$. For objects, we have
    \[ X\htimes_Z Y = \{(x,y,z,\alpha,\beta)\mid x\in X, y\in Y, z\in Z, \alpha\in\Hom(f(x),z), \beta\in\Hom(g(y),z) \}.\]
    Maps are the obvious commutative squares. There is a canonical equivalence
    \[ X\htimes_Z Y \simeq \{(x,y,\phi) \mid x\in X, y\in Y, \phi\in \Hom(f(x), g(y)) \} \]
    given by the map $(x,y,z,\alpha,\beta)\mapsto (x,y,\beta^{-1}\alpha)$.

    Given a map of pullback diagrams $\{X\to Z\gets Y\} \to \{X_1\to Z_1\gets Y_1\}$, if the maps $X\to X_1$, $Y\to Y_1$, $Z\to Z_1$ are equivalences, then so is the induced map $X\htimes_Z Y \to X_1\htimes_{Z_1} Y_1$ (Proposition \ref{prop:homotopy}).
%given a commutative diagram of groupoids
%    \[ \xymatrix{
%            X\ar[r]\ar[d]& Z\ar[d]&Y \ar[l]\ar[d] \\
%            X_1\ar[r]&Z_1&Y_1\ar[l]
%    }\]
%    if all of the vertical maps are weak equivalences, then the induced canonical map $X\htimes_{Z} Y \to X_1\htimes_{Z_1}Y_1$ is also a weak equivalence. 
    For a direct proof, see \cite[Tag 02XA]{Stacks}. 
    %Alternatively, observe that this is the so-called co-glueing Lemma for a right proper model category and $\Grpd$ has this structure with the given class of weak equivalences and fibrations \cite{H}. 
\end{example}

%Let $G$ be a group and let $X$ be a $G$-set. Write $\E_G X$ for the groupoid with set of objects $X$ and an arrow $x\mapright{g} y$ for each $g\in G$ such that $g(x) = y$. A map of $G$-sets $X\to Y$ (i.e., an equivariant map) induces a map $\E_GX\to \E_GY$. If $X\to Y$ is an isomorphism of $G$-sets, then $\E_GX \to \E_G Y$ is an isomorphism of groupoids. If $N\subset G$ is a normal subgroup that acts on $X$ freely, then the evident canonical map $\E_GX \to \E_{G/N}X/N$ is an acyclic fibration. For future use, we set
%    \[ \E G = \E_G G \qquad \text{and} \qquad \B G = \E_G\{\ast\},\]
%where $G$ acts on itself via multiplication.

\begin{example}\label{eg:hfix}
    Let $\Gamma$ be a group. We abuse notation and write $\Gamma$ for the category $\B\Gamma$. A $\Gamma$-diagram in $\Grpd$ amounts to a groupoid $X$ equipped with a $\Gamma$-action. In other words, $\Gamma$ acts on the sets of objects and morphisms of $X$ in a manner compatible with the structure maps of $X$ as a category. Set
    \[ X^{h\Gamma} = \holim_{\Gamma} X. \]
    Then $X^{h\Gamma}$ is the groupoid of \emph{homotopy fixed points}.
An object of $X^{h\Gamma}$ amounts to a pair $(x,\phi)$, where $x$ is an object of $X$ and $\phi$ is a rule that assigns, to each $g\in \Gamma$, a morphism
$\phi(g)\colon x \to gx$
in $X$ satisfying $\phi(1) = \id$ and
\begin{equation}\label{eq:cocycle} \phi(gh) = g\phi(h)\circ \phi(g),\end{equation}
for all $g,h\in\Gamma$.
An arrow $(x,\phi) \to (x_1,\phi_1)$ is a morphism $\alpha\colon x\to x_1$ in $X$ such that
\[ \phi_1(g)\circ \alpha = g\alpha\circ \phi(g), \]
for all $g\in \Gamma$. 
%If $f\colon X\to Y$ is an equivariant map between groupoids equipped with $\Gamma$-actions, then we obtain a morphism of groupoids $f_*\colon X^{h\Gamma}\to Y^{h\Gamma}$ via 
%\[f_*((x,\phi)) = (f(x),g\mapsto f(\phi(g))), \qquad f_*(\alpha) = f(\alpha). \]

The ordinary $\Gamma$-fixed points of $X$ are not invariant under equivalences: the map $\E\Gamma\to\ast$ is an acyclic fibration, but the $\Gamma$-action on $\E\Gamma$ is free.
Homotopy fixed points remedy this (Proposition \ref{prop:homotopy}).
A direct elementary proof for $\Gamma=\ZZ/2\ZZ$ may be found in \cite[Proposition 2.1.6]{Virk}. The same argument applies to an arbitrary group.
\end{example}

\begin{example}\label{eg:galois}Continue with the setup of the previous example. Suppose $G$ is a group on which $\Gamma$ acts via automorphisms. Let 
\[ Z^1(\Gamma; G) = \{\text{$G$-valued 1-cocycles of $\Gamma$}\}. \]
I.e., $Z^1(\Gamma; G)$ consists of functions $\sigma\colon \Gamma \to G$ satisfying
    \[ \sigma(gh) = \sigma(g)\cdot g\sigma(h) \]
    for all $g,h\in\Gamma$. Let $\E_GZ^1(\Gamma;G)$ be the groupoid with set of objects $Z^1(\Gamma; G)$, and an arrow $\sigma\to \sigma_1$ for each $\alpha\in G$ such that
    \[ \sigma_1(g) = \alpha\cdot \sigma(g) \cdot g\alpha^{-1}, \]
    for all $g\in\Gamma$. The set of isomorphism classes of objects in $\E_G Z^1(\Gamma; G)$ is the first non-abelian cohomology $H^1(\Gamma; G)$. The automorphism group of $\sigma\in \E_G Z^1(\Gamma; G)$ is 
    \[ K_{\sigma} = \{ \alpha\in G \mid \sigma(g)\cdot g\alpha\cdot \sigma(g)^{-1} = \alpha \text{ for all $g\in \Gamma$}\}.\]
    Thus, there is an equivalence
    \[ \E_GZ^1(\Gamma; G) \simeq \bigsqcup_{[\sigma]\in H^1(\Gamma; G)} \B K_{\sigma}, \]
    This equivalence depends on picking an object in each isomorphism class $[\sigma]$. In particular, there is usually no way to make it canonical.

    The $\Gamma$-action on $G$ yields a $\Gamma$-action on $\B G$. Objects in $(\B G)^{h\Gamma}$ amount to functions $\phi\colon \Gamma \to G$ satisfying $\phi(1) = 1$ and \eqref{eq:cocycle}. To such a $\phi$, assign a cocycle $\sigma_{\phi}$ by setting
   $\sigma_{\phi}(g)= \phi(g)^{-1}$,
   for all $g\in \Gamma$. This yields a canonical isomorphism
   \[ (\B G)^{h\Gamma} \mapright{\cong} \E_G Z^1(\Gamma; G). \]
   Hence, we obtain an equivalence
   \[ (\B G)^{h\Gamma} \simeq \bigsqcup_{[\sigma]\in H^1(\Gamma; G)} \B K_{\sigma}. \]
   In principle, this gives a presentation of $X^{h\Gamma}$ for an arbitrary groupoid $X$, since any groupoid is equivalent to one of the form $\bigsqcup_i \B G_i$. 
   %As before, this involves picking objects in isomorphism classes and usually cannot be made canonical.
\end{example}

\begin{example}\label{eg:inertia}
    %The following special case of the above needs to singled out.
    Let $\Gamma=\ZZ$ and let $X\in\Grpd$. Consider the trivial action of $\ZZ$ on $X$. Then $X^{h\ZZ}$ is called the \emph{free loop groupoid} of $X$, and is denoted $LX$. Its objects amount to
    \[ LX = \{ (x,\phi) \mid x\in X, \phi\in \Aut(x) \}. \]
    There is a canonical equivalence
    \[ LX \mapright{\sim} X \htimes_{X\times X} X, \]
    where the pullback is over the diagonal map $X\to X\times X$  (see \cite[Tag 04Z2]{Stacks}).
\end{example}

\begin{prop}\label{prop:homotopy}Let $\Gamma$ be a small category and let $f\colon X\to Y$ be a map of $\Gamma$-diagrams in $\Grpd$. If $f(\alpha)\colon X(\alpha) \to Y(\alpha)$ is an equivalence (resp. fibration) of groupoids for each $\alpha\in\Gamma$, then $f_*\colon \holim_{\Gamma}X \to \holim_{\Gamma} Y$ is an equivalence (resp. fibration). 
\end{prop}

\begin{proof}
    %This is essentially \cite[Chapter XI 5.5-5.6]{BK}. However, \cite{BK} is written in the language of simplicial sets. We orient the reader using the reference \cite{Hir} since our notation is closer to the latter.
    Let $B$ denote the nerve functor from the category of small categories to the category of simplicial sets. Then $B$ is full, faithful and commutes with small limits. Consequently, if $F$ and $G$ are small categories, we have a canonical isomorphism 
    %\cite[Section 4.4]{Gray} (\cite[Section 3]{T} has a nice exposition):
    \[ B\Map(F, G) \mapright{\cong}\Map(BF, BG), \]
where the right hand side is the simplicial function complex. 
Thus, $B\holim_{\Gamma}X$ is canonically isomorphic to the equalizer of
\[ 
\xymatrix{
    \underset{\alpha\in \Gamma}{\prod} \Map(B(\Gamma\downarrow \alpha), BX(\alpha))\ar@<-1ex>[r]_-{B\psi}\ar@<1ex>[r]^-{B\phi} & \underset{\{\alpha \to \beta\}\in \Gamma}{\prod}\Map(B(\Gamma\downarrow\alpha), BX(\beta))
} \]
and similarly for $B\holim_{\Gamma}Y$. This is the homotopy limit formula of \cite[Definition 18.1.8]{Hir} for simplicial sets. 
As $B$ preserves and reflects weak equivalences as well as fibrations, we are reduced to showing that
%(see \cite[Section 5]{A})
\[ Bf_*\colon \holim_{\Gamma}BX \to \holim_{\Gamma}BY \]
is a weak equivalence/fibration of simplicial sets in our situation.
As each $X(\alpha)$ and each $Y(\alpha)$ is a groupoid, it follows that each $BX(\alpha)$ and each $BY(\alpha)$ is a Kan complex. 
Now the desired statement regarding weak equivalences is \cite[Theorem 18.5.3 (2)]{Hir}, and that for fibrations is \cite[Theorem 18.5.1 (2)]{Hir}.
\end{proof}

\begin{prop}\label{prop:commute}Let $\Gamma$ and $\Delta$ be small categories and let $X$ be a $\Gamma\times\Delta$-diagram in $\Grpd$. Then we have canonical isomorphisms
    \[ \holim_{\Gamma}\holim_{\Delta}X \cong \holim_{\Gamma \times \Delta}X \cong \holim_{\Delta}\holim_{\Gamma}X.\]
\end{prop}

\begin{proof}
%    Use the nerve functor as in the proof of Proposition \ref{prop:homotopy} and invoke \cite[Chapter XI 4.3]{BK}.
    For a small category $K$, the functor $\Map(K, - )$ commutes with limits and satisfies the exponential law. Further, limits always commute with each other. These statements applied to our formula for the homotopy limit give the desired result.
\end{proof}

\subsection{Filtered colimits}All small colimits exist in $\Grpd$. 
%(see \cite[Chapter 9]{Hig} or \cite[Appendix A]{H})
Unlike limits, colimits can be nebulous to describe explicitly. However, filtered colimits present no difficulties: they are given by filtered colimits of sets on objects and morphisms.

\begin{lemma}\label{lem:colfinite}
    Filtered colimits of groupoids commute with finite products of groupoids.
\end{lemma}

\begin{proof}Immediate from the same for the category of sets.
\end{proof}

\begin{lemma}\label{lem:key}
    Let $I$ be a small filtered category and let $X$ be an $I$-diagram in $\Grpd$. If $K$ is a finite category, then the canonical map
    \[ \varinjlim_I\Map(K, X(i)) \to \Map(K, \varinjlim_I X(i)) \]
    is an isomorphism.
\end{lemma}

\begin{proof}
    The inverse is defined as follows. As $K$ is finite, an object of $\Map(K, \varinjlim_I X(i))$ amounts to picking a finite number of objects, and a finite number of morphisms satisfying a finite number of relations in $\varinjlim_I X(i)$. As $I$ is filtered, we may assume that these objects, morphisms and relations come from some $X(j)$. This yields an object of $\varinjlim_I \Map(K, X(i))$ independent of the choices made. Similarly for morphisms.
\end{proof}

\begin{prop}\label{prop:colimits}Let $I$ be a small filtered category and let $\Gamma$ be a finite category. Let $X$ be a $I\times \Gamma$-diagram in $\Grpd$. Then the canonical map
    \[\varinjlim_I \holim_{\Gamma}X \to \holim_{\Gamma}\varinjlim_I X \]
    is an isomorphism.
\end{prop}

\begin{proof}Apply Lemma \ref{lem:colfinite} and Lemma \ref{lem:key}.
\end{proof}

\subsection{Stacks}
Let $\cC$ be a Grothendieck site with enough points. 
By definition of a point, the stalk of any sheaf of sets at a point is a left exact functor to the category of sets. On the other hand, a stalk is defined as a certain (small) colimit of sets (see \cite[Tag 00Y3]{Stacks}). As a small category $I$ is filtered if and only if colimits into the category of sets, indexed by $I$, commute with finite limits, 
%\cite[Theorem 3.1.6]{KS}, 
it follows that stalks are given by filtered colimits. This is also visibly obvious for most commonly used sites (the \'etale site on varieties, open subsets of a sober topological space, etc.).
Set
\[ \PreGrpd(\cC) = \text{category of presheaves of groupoids on $\cC$}.\]
Note: `presheaf' means `strict presheaf' (as opposed to lax). 
%For a point $t$ of $\cC$, and $X\in\PreGrpd(\cC)$, write $X_t$ for the stalk at $t$.
We will not distinguish between an object $X\in\cC$ and the presheaf $\Hom_{\cC}(-, X)$ it represents.

A stalk for an object of $\PreGrpd(\cC)$ is defined using the same colimit formula as for sheaves of sets.
A map $X\to Y$ in $\PreGrpd(\cC)$ is called a \emph{local weak equivalence} (resp. \emph{local fibration}) if it induces an equivalence (resp. fibration) of groupoids on all stalks. The map is a \emph{sectionwise weak equivalence} if it induces an equivalence of groupoids $X(U) \to Y(U)$ for all $U\in \cC$. A sectionwise weak equivalence is always a local weak equivalence. However, the converse does not generally hold.

%In general, `local' will refer to a property required to hold on all stalks. `Sectionwise' will refer to a property required to hold on all sections.
%
An object in $\PreGrpd(\cC)$ is a stack if it is a sheaf of groupoids that satisfies effective descent.
A map between stacks is a local weak equivalence if and only if it is a sectionwise weak equivalence \cite[Proposition 9.28]{J}. In other words, local weak equivalence for stacks amounts to the usual notion of equivalence of stacks. Analogous to the classical process of sheafification, a presheaf of groupoids may be sheafified and then further completed to a stack in a canonical fashion. That is, for each $X\in \PreGrpd(\cC)$, there exists a stack $[X]$ along with a canonical local weak equivalence $X\mapright{\sim} [X]$ (see \cite[Corollary 9.27]{J} or \cite[Tag 02ZO]{Stacks} for details). The stack $[X]$ is called the stack completion or stackification of $X$. For ease of reference, 
\[ [X] = \text{stack completion of $X$.} \]
If $\Gamma$ is a small category and $X$ is a $\Gamma$-diagram in $\PreGrpd(\cC)$, we will write $[X]$ for the diagram obtained by composing with stack completion.
%
%\begin{example}A morphism between sheaves of sets (viewed as groupoids) is a local weak equivalence if and only if it is an isomorphism.
%\end{example}
%
%\begin{example}If $\cC=\ast$, then $\PreGrpd(\cC)$ is the category of groupoids. A local weak equivalence amounts to an equivalence of categories.
%\end{example}

\begin{example}\label{eg:BG}Let $\cC$ be the big \'etale site of complex varieties.\footnote{`Complex variety' = `separated scheme of finite type over $\Spec(\CC)$'.} Let $G$ be a linear algebraic group. Define $BG \in \PreGrpd(\cC)$ by $BG(U)=\B G(U)$ for all $U\in\cC$. I.e., $BG(U)$ consist of a single object with automorphism group $G(U)$. Let $[BG]$ denote the classifying stack of principal $G$-bundles. The map $BG\to [BG]$ sending the single object in $BG(U)$ to the trivial bundle on $U$ is a local weak equivalence (\'etale local triviality of principal bundles). However, it is not a sectionwise weak equivalence.
\end{example}

\subsection{Homotopy limits: stacks}
Let $\Gamma$ be a small category and let $X$ be a $\Gamma$-diagram in $\PreGrpd(\cC)$. Define the homotopy limit $\holim_{\Gamma}X$ sectionwise, i.e., for $U\in \cC$:
\[ (\holim_{\Gamma}X)(U) = \holim_{\Gamma}X(U). \]

\begin{example}As in Example \ref{eg:homotopypullback}, if $\Gamma$ is the pullback category, we write $X\htimes_Z Y$ for the corresponding homotopy limit and call it the \emph{homotopy pullback}. If all the objects involved in the pullback diagram, $\cX,\cY,\cZ$, are stacks, then
    \[ \cX\htimes_{\cZ} \cY \simeq \text{(2-)fibre product of stacks}. \]
\end{example}

\begin{example}
    If $\Gamma$ is a group then, as in Example \ref{eg:hfix}, we set $X^{h\Gamma} = \holim_{\Gamma} X$ and call it the \emph{homotopy fixed points} of the $\Gamma$-action. If $\cX$ is a stack (equipped with a $\Gamma$-action), then the definition of $\cX^{h\Gamma}$ amounts to that of fixed point stack in \cite{R}:
    \[ \cX^{h\Gamma} \simeq \text{$\Gamma$-fixed point stack}. \]
\end{example}

\begin{example}
    As in Example \ref{eg:inertia}, if $X \in \PreGrpd(\cC)$, we write $LX$ for the homotopy fixed points of $\ZZ$ acting on $X$ trivially. If $\cX$ is a stack, this amounts to the inertia stack 
    \[ L\cX \simeq \text{inertia stack of $\cX$}. \]
\end{example}

\begin{prop}\label{prop:stackinv}Let $\Gamma$ be a small category and let $\cX$ and $\cY$ be $\Gamma$-diagrams of stacks (i.e., $\cX(\alpha)$ and $\cY(\alpha)$ are stacks for each $\alpha\in \Gamma$). Let $f\colon \cX\to\cY$ be a natural transformation of diagrams. If $f(\alpha)\colon \cX(\alpha) \to \cY(\alpha)$ is an equivalence of stacks for each $\alpha\in\Gamma$, then $f_*\colon \holim_{\Gamma}\cX \to \holim_{\Gamma}\cY$ is a sectionwise weak equivalence.
\end{prop}

\begin{proof}Apply Proposition \ref{prop:homotopy}.
\end{proof}

\begin{thm}\label{thm:stack}Let $\Gamma$ be a small category and let $\cX$ be a $\Gamma$-diagram of stacks. Then $\holim_{\Gamma}\cX$ is a stack.
\end{thm}

\begin{proof}
    %One may apply Proposition \ref{prop:homotopy} and Proposition \ref{prop:commute} to the characterization of stacks given in \cite[Theorem 1.1]{H} (the characterization is in terms of sectionwise weak equivalences with certain auxilliary homotopy limits associated to every cover in $\cC$; Proposition \ref{prop:commute} allows us to commute these with $\holim_{\Gamma}$, and then Proposition \ref{prop:homotopy} applies). 
    %
    Use the injective model structure of \cite[Proposition 9.19]{J}. Let $Y\mapright{\sim} RY$ denote a canonical injective fibrant model for $Y\in\PreGrpd(\cC)$. In particular, $Y\mapright{\sim} RY$ is a local weak equivalence and $RY$ is injective fibrant.
    Every injective fibrant object is a stack \cite[Remark 9.23, Proposition 9.28]{J}. 
    Thus, if each $\cX(\alpha)$ is a stack, then $\holim_{\Gamma}\cX \to \holim_{\Gamma}R\cX$ is a sectionwise weak equivalence by Proposition \ref{prop:stackinv}. Further, homotopy limits preserve the class of injective fibrant objects by \cite[Theorem 18.5.2 (2)]{Hir}.
    So $\holim_{\Gamma}R\cX$ is a stack. Hence, $\holim_{\Gamma}\cX$ is a stack. 
\end{proof}

\begin{thm}\label{thm:commute}Let $\Gamma$ and $\Delta$ be small categories and let $X$ be a $\Gamma\times\Delta$-diagram in $\PreGrpd(\cC)$. Then we have canonical isomorphisms
    \[ \holim_{\Gamma}\holim_{\Delta}X \cong \holim_{\Gamma \times \Delta}X \cong \holim_{\Delta}\holim_{\Gamma}X.\]
\end{thm}

\begin{proof}
    Immediate from Proposition \ref{prop:commute}.
\end{proof}

\begin{cor}\label{cor:pullback}Let $\cX,\cY,\cZ$ be stacks equipped with the action of a group $\Gamma$. Suppose $\cX\to \cZ$ and $\cY\to \cZ$ are $\Gamma$-equivariant maps. Then the stacks $(\cX\htimes_{\cZ} \cY)^{h\Gamma}$ and $\cX^{h\Gamma}\htimes_{\cZ^{h\Gamma}} \cY^{h\Gamma}$ are canonically equivalent.
\end{cor}

\begin{cor}\label{cor:inertia}Let $\cX$ be a stack equipped with the action of a group $\Gamma$. Then the stacks $(L\cX)^{h\Gamma}$ and $L\cX^{h\Gamma}$ are canonically equivalent.
\end{cor}

\begin{thm}\label{thm:main}
    Let $\Gamma$ be a finite category and let $f\colon X\to Y$ be a map of $\Gamma$-diagrams in $\PreGrpd(\cC)$.
    If each $f(\alpha)\colon X(\alpha) \to Y(\alpha)$, $\alpha\in\Gamma$, is a local weak equivalence (resp. local fibration), then the induced map 
            \[f_*\colon \holim_{\Gamma} X \to \holim_{\Gamma} Y\]
            is a local weak equivalence (resp. local fibration).
\end{thm}

\begin{proof}
    Apply Proposition \ref{prop:homotopy} and Proposition \ref{prop:colimits}.
\end{proof}

\begin{thm}\label{thm:presentation}Let $\Gamma$ be a finite category and let $X$ be a $\Gamma$-diagram in $\PreGrpd(\cC)$. Then the stacks $[\holim_{\Gamma}X]$ and $\holim_{\Gamma}[X]$ are canonically equivalent.
\end{thm}

\begin{proof}
    The canonical map $X(\alpha)\to [X(\alpha)]$ is a local weak equivalence for each $\alpha$. So, if $\Gamma$ is a finite category, $\holim_{\Gamma}X \to \holim_{\Gamma}[X]$ is a local weak equivalence by Theorem \ref{thm:main}. As $\holim_{\Gamma}[X]$ is already a stack by Theorem \ref{thm:stack}, we are done.
\end{proof}

\begin{cor}\label{cor:hfixstack}
    Let $\Gamma$ be a finite group acting on $X\in\PreGrpd$. Then the stacks $[X^{h\Gamma}]$ and $[X]^{h\Gamma}$ are canonically equivalent.
\end{cor}

\begin{thm}\label{thm:lastbeforeeg}Let $I$ be a small filtered category and let $\Gamma$ be a finite category. Let $X$ be a $I\times \Gamma$-diagram in $\PreGrpd(\cC)$. Then the canonical map
    \[ \varinjlim_I \holim_{\Gamma} X \to \holim_{\Gamma}\varinjlim_I X \]
    is an isomorphism.
\end{thm}

\begin{proof}
    Apply Proposition \ref{prop:colimits}.
\end{proof}

\begin{example}\label{eg:final}
    Let $\cC$ be the big \'etale site of complex varieties. 
    Let $\Gamma$ be a finite group acting on a linear algebraic group $G$ via automorphisms. 
    We will use the notation of Example \ref{eg:galois} and Example \ref{eg:BG}. In particular, $BG(U)$ denotes the groupoid consisting of a single object with automorphisms $G(U)$, for all $U\in\cC$; $[BG]$ is its stack completion - the stack of principal $G$-bundles; $Z^1(\Gamma; G(U))$ is the set of 1-cocycles, etc.
    Then $\Gamma$ acts on $BG$ and we have a canonical isomorphism
    \[ (BG)^{h\Gamma} \mapright{\cong} E_G Z^1(\Gamma; G), \]
    where, for $U\in \cC$, in the notation of Example \ref{eg:hfix},
    \[ E_G Z^1(\Gamma; G)(U) = \E_{G(U)} Z^1(\Gamma; G(U)).\]
    As $\Gamma$ is finite, and in particular has a finite presentation, the presheaf $U\mapsto Z^1(\Gamma; G(U))$ is representable by an affine $G$-variety $Z$. 
    Explicitly, $Z$ is the variety of group homomorphisms $\Gamma \to G\rtimes \Gamma$ that are sections of the canonical map $G\rtimes \Gamma \to \Gamma$.

    Embed $G\rtimes \Gamma$ in some $GL_n$ and consider the representation variety $Y$ of group homomorphisms $\Gamma \to GL_n$. Every representation of $\Gamma$ over $\CC$ is completely reducible. Further, for a fixed dimension there are finitely many isomorphism classes of these.
    So, under the conjugation action of $GL_n$, the variety $Y$ has finitely many orbits and each orbit is closed \cite[Theorem 2.17]{PR}. On the other hand, $Z$ embeds into $Y$, and each $G$-orbit in $Z$ is the intersection of some $GL_n$-orbit with $Z$ (see \cite[Lemma 2.11 and Theorem 2.17]{PR}). Thus, there are finitely many $G$-orbits in $Z$ and they are all closed. Consequently, $Z$ is isomorphic to the disjoint union of its orbits.
These orbits are parametrized by $H^1(\Gamma; G(\CC))$ (see Example \ref{eg:hfix}). Moreover, for a geometric point $\sigma \in Z^1(\Gamma; G(\CC))$, the stabilizer of $\sigma$ is
    \[ K_{\sigma} = \{ \alpha\in G \mid  \sigma(g)\cdot g\alpha\cdot \sigma(g)^{-1} = \alpha \text{ for all $g\in\Gamma$} \}.\]
    Thus, we have a local weak equivalence
    \[ (B G)^{h\Gamma} \simeq \bigsqcup_{[\sigma]\in H^1(\Gamma; G(\CC))} B K_{\sigma}.\]
    As in Example \ref{eg:hfix}, it is generally impossible to make this canonical.

    Write $[Z/G]$ for the stack completion of $E_G Z^1(\Gamma; G)$, i.e., the stack quotient of $Z$ by $G$. By Corollary \ref{cor:hfixstack}, we have equivalences of stacks:
    \[ [BG]^{h\Gamma} \simeq [Z/G] \simeq \bigsqcup_{[\sigma] \in H^1(\Gamma; G(\CC))} [BK_{\sigma}].\]
    The first of these equivalences is canonical. However, the second depends on the various choices made above and cannot usually be made canonical. In particular, it is generally impossible to identify the first and last spaces here. 
Readers may explore the consequences of Corollary \ref{cor:pullback} and Corollary \ref{cor:inertia} in this setting at their leisure.
\end{example}

\subsection{Derived functors}
Given a small category $\Gamma$, write $\PreGrpd_{\Gamma}(\cC)$ for the category of $\Gamma$-diagrams in $\PreGrpd(\cC)$. Call a natural transformation of $\Gamma$-diagrams $X\to Y$ an objectwise local weak equivalence if $X(\alpha)\to Y(\alpha)$ is a local weak equivalence for each $\alpha\in \Gamma$.
Let $\Ho(\cC)$ be the localization of $\PreGrpd(\cC)$ with repect to local weak equivalences, and let $\Ho_{\Gamma}(\cC)$ be the localization of $\PreGrpd_{\Gamma}(\cC)$ with respect to objectwise local weak equivalences. Let
\[ \iota\colon \PreGrpd(\cC) \to \PreGrpd_{\Gamma}(\cC) \]
be the functor which assigns to $X\in \PreGrpd(\cC)$ the constant $\Gamma$-diagram with value $X$ (all morphisms in $\Gamma$ are mapped to the identity on $X$). 
Then $\iota$ is left adjoint to
\[ \varprojlim_{\Gamma}\colon \PreGrpd_{\Gamma}(\cC) \to \PreGrpd(\cC).\]
Certainly, $\iota$ sends local weak equivalences to objectwise local weak equivalences. Hence, it induces a functor $\iota\colon \Ho(\cC) \to \Ho_{\Gamma}(\cC)$. However, since $\varprojlim_{\Gamma}$ does not usually map objectwise local weak equivalences to local weak equivalences, it does not immediately induce a functor $\Ho_{\Gamma}(\cC) \to \Ho(\cC)$. 

\begin{thm}The functor $X \mapsto \holim_{\Gamma}[X]$ is right adjoint to $\iota\colon \Ho(\cC) \to \Ho_{\Gamma}(\cC)$. In particular, there is a canonical isomorphism
    \[ \R\varprojlim_{\Gamma} X \cong \holim_{\Gamma}[X] \]
    in $\Ho(\cC)$, where $\R\varprojlim_{\Gamma}$ is the total right derived functor of $\varprojlim_{\Gamma}$ in the sense of \cite{Q}.
\end{thm}

\begin{proof}
    Use the injective model structure on $\PreGrpd(\cC)$ from \cite[Proposition 9.19]{J}.
    In particular, weak equivalences in this model structure are local weak equivalences (the slightly confusing terminology is unavoidable). This induces a model structure on $\PreGrpd_{\Gamma}(\cC)$ as follows.
    Fibrations in $\PreGrpd_{\Gamma}(\cC)$ are those natural transformations that are objectwise injective fibrations. Similarly for weak equivalences. Cofibrations in $\PreGrpd_{\Gamma}(\cC)$ are maps that have the left lifting property with respect to maps which are both fibrations and weak equivalences. As $\PreGrpd(\cC)$ is cofibrantly generated \cite[Theorem 5.8]{J}, this defines a model structure on $\PreGrpd_{\Gamma}(\cC)$ by \cite[Theorem 11.6.1]{Hir}. 
    Consider the functor 
    \[ \tilde \iota\colon \PreGrpd(\cC) \to \PreGrpd_{\Gamma}(\cC) \]
    that assigns to $X\in \PreGrpd(\cC)$ the diagram
    \[ \alpha\mapsto \mu(\Gamma\downarrow \alpha) \times X, \]
    where $\mu(K)$ denotes the category obtained by inverting all morphisms in $K$ (i.e., the left adjoint to the forgetful functor from $\Grpd$ to the category of small categories), and $\mu(K)\times X$ is a sectionwise prescription:
    \[ (\mu(K)\times X)(U) = \mu(K) \times X(U),\]
    for all $U\in\cC$. Then $\tilde \iota$ is left adjoint to $\holim_{\Gamma}$.
    By \cite[Proposition 18.5.1 (2)]{Hir}, $\holim_{\Gamma}$ preserves fibrations as well as the class of fibrations that are also weak equivalences.
Hence, the total left derived functor
\[ \L \tilde \iota \colon \Ho(\cC) \to \Ho_{\Gamma}(\cC) \]
and the total right derived functor
\[ \R \holim_{\Gamma}\colon \Ho_{\Gamma}(\cC) \to\Ho(\cC) \]
exist and form an adjoint pair by \cite[Proposition 8.5.3 (3), Proposition 8.5.7]{Hir} and \cite[Chapter 1.4, Theorem 3]{Q}. As $\tilde \iota$ preserves weak equivalences, $\L \tilde \iota = \tilde \iota$. Since each overcategory $(\Gamma\downarrow \alpha)$ has a final object, the canonical map $\mu(\Gamma\downarrow \alpha)\to \ast$ is an equivalence of groupoids. It follows that $\tilde \iota$ and  $\iota$ are canonically isomorphic as functors $\Ho(\cC)\to\Ho_{\Gamma}(\cC)$.
Hence, it only remains to show that $\R\holim_{\Gamma}$ has the requisite form. By construction,
\[ \R\holim_{\Gamma} X = \holim_{\Gamma}(RX), \]
where $RX$ denotes the diagram obtained from $X$ by replacing each $X(\alpha)$ with a canonical injective fibrant model $RX(\alpha)$. On the other hand, we have a canonical isomorphism in $\Ho(\cC)$:
\[ \holim_{\Gamma}(RX) \cong \holim_{\Gamma}[X]. \qedhere \]
\end{proof}

\end{document}